\documentclass[10pt]{article}
\usepackage{amssymb}
\usepackage{graphicx}
\usepackage{amsmath}

\newtheorem{theorem}{Theorem}

\newtheorem{corollary}[theorem]{Corollary}

\newtheorem{definition}[theorem]{Definition}

\newtheorem{lemma}[theorem]{Lemma}

\newtheorem{proposition}[theorem]{Proposition}
\newtheorem{remark}[theorem]{Remark}

\newenvironment{proof}[1][Proof]{\textbf{#1.} }{\ \rule{0.5em}{0.5em}}
\begin{document}

\title{
$\mathcal{D}$%
-modules 
on a class of 
$G$%
-representations}
\author{Philibert Nang}
\maketitle

\begin{abstract}
We give an answer to the abstract Capelli problem:
Let
$(G,V)$ 
be a multiplicity-free finite-dimensional representation of a connected reductive complex Lie Group
$G$
and
$G'$
be its derived subgroup.
Assume that the categorical quotient
$V//G$
is one-dimensional
, i.e., there exists a polynomial
$f$ 
generating the algebra of
$G'$%
-invariant 
polynomials on 
$V$
(%
$\mathbb{C}[V]^{G'}= \mathbb{C}[f]$%
)
and such that
$f \not\in \mathbb{C}[V]^{G}$%
. 
We prove that the category of regular holonomic 
$\mathcal{D}_{V}$%
-modules invariant under the action of
$G$
is equivalent to the category of graded modules of finite type over a suitable algebra
$\mathcal{A}$%
. 
This has been conjectured by T. Levasseur \cite[Conjecture 5.17, p. 508]{L}
(after we had already proved it in some cases: \cite{N1}, \cite{N2}, \cite{N3}, \cite{N4}).\medskip

\noindent \textbf{Keywords}\textit{: }$\mathcal{D}$%
-modules, 
holonomic $\mathcal{D}$%
-modules, invariant differential operators, irreducible representations,
prehomogenous vector spaces, multiplicity-free spaces, Capelli identity, representations of Capelli type \textit{.\medskip }

\end{abstract}

\section{Introduction}

One of the hightlights of the theory of
$\mathcal{D}$%
-modules 
is the Riemann-Hilbert correspondence. It establishes a bridge between analytic objects
(regular holonomic 
$\mathcal{D}$%
-modules) and geometric ones (constructible sheaves). This has been proved independently by 
M. Kashiwara \cite{K1} and Z. Mebkhout \cite{Me}.
One interprets the Riemann-Hilbert correspondence as a generalization of Deligne's
solution of the
21-st problem of Hilbert
(see N. M. Katz \cite{Kat} and also Z. Mebkhout \cite{Me1}, \cite{Me2}).
In a sense the Riemann-Hilbert correspondence generalizes the well known one-to-one correspondence
between vector bundles with integrable connections and local systems. Actually, this correspondence
induces an equivalence of categories between the category of regular holonomic 
$\mathcal{D}$%
-modules
with a characteristic variety
$\Lambda$
and that of perverse sheaves on
$V$
(where
$V$
is a complex manifold)
with microsupport
$\Lambda$%
.
In fact, a perverse sheaf is not really a sheaf but rather a complex of sheaves.
More precisely, it is an object of the derived category. Also, the notion of a morphism
between perverse sheaves is difficult to handle. But, on the other side, it is clear what
is meant by a 
$\mathcal{D}_{V}$%
-module
and a 
$\mathcal{D}_{V}$%
-linear morphism.
So, if one wants to understand the structure of perverse sheaves on
$V$%
, it is certainly worthwhile to take advantage of the Riemann-Hilbert correspondence and
to study the category of regular holonomic
$\mathcal{D}_{V}$%
-modules. These objects are perhaps more accessible.\newline

\qquad 
Let 
$ G $
be a complex connected reductive algebraic group,
and let
$ G'=\left[G, G\right] $
be its derived subgroup.
Denote by
$(G, \rho, V)$
or
$ \left(G, V\right) $
a rational finite-dimensional linear representation of
$G$
($\rho : G \longrightarrow GL(V, \mathbb{C})$) 
and 
$\mathbb{C}[V]$ 
the algebra of polynomials on
$V$%
.
The action of 
$G$
on 
$V$
extends to
$\mathbb{C}[V]$%
.
We will denote by
$\mathbb{C}[V]^{G}  \subset \mathbb{C}[V]$
the subalgebra of
$G$%
-invariant polynomials on
$V$%
.
We assume that
$ \left(G, V\right) $
is a multiplicity-free space, that is, the associated representation
of
$G$
on
$\mathbb{C}[V]$
decomposes without multiplicities. In other words, each irreducible representation of
$G$
occurs at most once in
$\mathbb{C}[V]$
(see definition \ref{d00}).
For the classification and properties of multiplicity-free spaces, we refer to
the work by C. Benson and G. Ratcliff \cite{BR}, F. Knops \cite{Kn}, A. Leahy\cite{Le}.
We assume furthermore that the multiplicity-free space
$ \left(G, V\right) $
has a one-dimensional quotient (i.e., the categorical quotient is one-dimensional: 
$\dim\left(V//G\right) = 1$%
), that is, there exists a polynomial
$f$  
on 
$V$
such that the subalgebra 
$\mathbb{C}[V]^{G'} $ 
of
$G'$%
-invariant 
polynomials on 
$V$
is the algebra of polynomials in
$f$
(i.e.,
$\mathbb{C}[V]^{G'}= \mathbb{C}[f]$%
), and such that
$f \not\in \mathbb{C}[V]^{G}$
(see definition \ref{d2}).
Then, it is known that:
$G$
acts on 
$V$
with an open orbit, and in this case the representation
$(G,V)$
is called a prehomogeneous vector space
(see M. Sato \cite{S}, \cite{S-K} or T. Kimura \cite[chap. 2]{Ki}). 
Moreover, it is shown in \cite[p. 39, proposition 2.22 ]{Ki} that: for 
such a reductive prehomogeneous vector space, there exists a constant 
coefficient differential operator 
$ \Delta $
and a polynomial
$$ b(s)=c(s+1)(s+\lambda_{1} +1)\cdots(s+\lambda_{n-1} +1) \in \mathbb{R}_{n}[s]\text{,}\quad c>0\text{,} $$
called the Bernstein-Sato polynomial of
$ f $
such that
$$ \Delta{f^{s+1}}= b(s)f^{s} \text{.}$$
M. Kashiwara \cite{K0} has shown that the roots of this polynomial are rational, i.e.,
$\lambda_{j} \in \mathbb{Q}$
for 
$1 \leq j \leq n-1$%
.

As usual 
$\mathcal{D}_{V}$
is the sheaf of rings of differential operators on
$V$ 
with holomorphic coefficients.
Let us now point out that the action of
$G$
on
$\mathbb{C}[V]$
extends to 
$\Gamma{(V,\mathcal{D}_{V})^{\mathrm{pol}}}$
the 
$\mathbb{C}$%
-algebra of differential operators on
$V$
with polynomial coefficients in 
$\mathbb{C}[V]$%
.
This gives rise to a natural algebra: the Weyl algebra
$\Gamma \left( V,\mathcal{D}_{V}\right) ^{G}$ 
of polynomial coefficients
$G$%
-invariant differential operators on
$V$%
.\newline

\noindent If
$G$
is a Lie group, 
let 
$ \mathfrak{g} $
be the Lie algebra of
$ G $
and 
$ U(\mathfrak{g}) $ 
be the universal enveloping algebra of
$ \mathfrak{g} $%
.
A representation as above
$ \left(G, V\right) $
is said to be of "Capelli type" if 
$ \left(G, V\right) $
is an irreducible multiplicity-free representation (MF for short) such that:
the subalgebra of
$G$%
-invariant global algebraic sections 
$\Gamma \left( V,\mathcal{D}_{V}\right) ^{G}$
is the image of
$Z\left(U\left(\mathfrak{g}\right)\right)$%
, the center of
$ U(\mathfrak{g}) $%
, under the differential
$ \tau:\mathfrak{g}\longrightarrow \Gamma{(V,\mathcal{D}_{V})^{\mathrm{pol}}} $
of the 
$ G $%
-action,
i.e.,
$$\tau\left(Z\left(U\left(\mathfrak{g}\right)\right)\right)= \Gamma\left( V,\mathcal{D}_{V}\right) ^{G} $$
(see definition \ref{d3}).
Note that these representations have been studied by R. Howe and T. Umeda in 
\cite{H-U1},\cite{U}:
they fall into eight cases (see Appendix).

If 
$ \left(G, V\right) $
is of Capelli type; in particular if
$ \left(G, V\right) $
is MF, then V. G. Kac
\cite{Ka} asserts that 
$ G $
has finitely many orbits
$(V_{k})_{k\in K}$%
.
Let us denote by
$\Lambda :=\bigcup\limits_{k\in K}\overline{T_{V_{k}}^{\ast }V} \subset T^{\ast }V$
the Lagrangian subvariety which is the union of the closure of conormal bundles to the 
$G$%
-orbits
(see \cite{Pa}).
\newline
Recall that a coherent
$\mathcal{D}_{V}$%
-module
$\mathcal{M}$
is said to be holonomic if its characteristic variety 
$\rm{char}\left(\mathcal{M}\right)$ 
is Lagrangian. Equivalently, the characteristic variety is of dimension equal to 
$\dim{V}$%
.
The holonomic 
$\mathcal{D}_{V}$%
-module
$\mathcal{M}$
is called regular if there exists a global good filtration
$F\mathcal{M}$
on
$\mathcal{M}$
such that the annihilator of 
$\mathrm{gr}^{F}\mathcal{M}$
(i.e., the ideal
$\mathrm{ann}_{\mathbb{C}[T^{*}V]}\mathrm{gr}^{F}\mathcal{M}$) 
is a radical ideal in
$\mathrm{gr}^{F}\mathcal{M}$
(see \cite[definition 5.2]{K1} or 
\cite[Corollary 5.1.11]{KK1}).
\newline
Denote by 
$\mathrm{Mod}_{\Lambda }^{\mathrm{rh}}(\mathcal{D}_{V})$ 
the full category whose objects are holomorphic regular holonomic 
$\mathcal{D}_{V}$%
-modules
$\mathcal{M}$%
, whose characteristic variety
$\rm{char}\left(\mathcal{M}\right)$ 
is contained in
$\Lambda$%
, equivalently those which admit global good filtrations
stable under the induced action of the Lie algebra
$\mathfrak{g}$
of
$G$
on
$\mathcal{M}$
(see Remark \ref{special})
. The general problem consists in the description of the category
$\mathrm{Mod}_{\Lambda }^{\mathrm{rh}}(\mathcal{D}_{V})$%
.\newline

The expected shape to the general solution of the family of problems is as follows. 
Let us first recall that
$ G' $
is the derived subgroup of
$ G $%
. We denote by
$$\bar{\mathcal{A}}:= \Gamma \left( V,\mathcal{D}_{V}\right) ^{G'} \subset \Gamma{(V,\mathcal{D}_{V})^{\mathrm{pol}}}$$
the 
$\mathbb{C}$%
-algebra formed by
$G'$%
-invariant global algebraic sections of
$\mathcal{D}_{V}$%
, i.e.
the algebra of polynomial coefficients
$ G' $%
-invariant differential operators. This algebra is well understood (see \cite{H-U1}, \cite{L}),
in particular it contains 
$\theta$
the Euler vector field on
$V$%
. Note that R. Howe and T. Umeda
\cite{H-U1}
have proved that when
$ (G, V) $
is of Capelli type, the algebra
$\bar{\mathcal{A}}$
is a polynomial algebra on a canonically defined set of generators. These generators are
precisely the Capelli operators.
T. Levasseur \cite[Theorem 4.11, p. 491]{L}, 
H. Rubenthaler \cite[p. 1346, proposition 3.1, 1)]{R1}
or
\cite[p. 24, theorem 5.3.3]{R2}
and
Z. Yan \cite[theorem 1.9]{Y}
gave a general description of this algebra.
We should also mention the contribution by M. Muro, in the real case 
$(G, V) = (GL(n, \mathbb{R}), S^{2}(\mathbb{R}^{n}))$
in
\cite[Proposition 2.1, p. 356]{Mu}. 
Finally, when
$(G, V) = (GL(n, \mathbb{C})\times SL(n, \mathbb{C}), M_{n}(\mathbb{C}))$%
,
$(GL(2m, \mathbb{C}),\; \Lambda^{2} \mathbb{C}^{2m})$%
,
$(GL(n, \mathbb{C}),\; S^{2}\mathbb{C}^{n} )$%
, the author obtained a concrete description with explicit relations in
\cite[Proposition 6, p. 120 ]{N0}, \cite[Proposition 5, p. 637-638 ]{N1}, \cite[Proposition 8, p. 4]{N5}. \newline

\noindent If 
$\mathcal{J}:=\mathrm{ann}\mathbb{C}[V]^{G'} = \mathrm{ann}\mathbb{C}[f] \subset \bar{\mathcal{A}}$
denotes the two sided ideal annihilator of
$G'$%
-invariant polynomials on
$V$%
,
we consider
$ \mathcal{A} $
the quotient algebra
$\bar{\mathcal{A}}/ \bar{\mathcal{J}}$%
, going modulo a suitable ideal
$\bar{\mathcal{J}}$
of 
$\bar{\mathcal{A}}$
described in section \ref{2}:
$\bar{\mathcal{J}}$
is the preimage in 
$\overline{\mathcal{A}}$
of the ideal in
$\overline{\mathcal{A}}/\mathcal{J}$
defined by specific relations (\ref{a0}), (\ref{b0}), (\ref{c0}), (\ref{d0})
of Proposition \ref{pu}.
Following the work by Benson - Ratcliff \cite{BR}, Howe - Umeda \cite{H-U1}, Knopp \cite{Kn}
and Levasseur \cite{L},
we will deduce that the quotient algebra 
$\mathcal{A}$
is generated by the following three operators and relations (see Corollary \ref{cu}): 
$\theta$
the Euler vector field on
$V$%
, 
$f$
the multiplication by the polymonial
$f(x)$ 
of degree
$n$%
,
and the differential operator
$\Delta := f\left(\dfrac{\partial}{\partial{x}}\right)$
as above satisfying the Bernstein-Sato equations:
$$ \Delta{f}= c(\dfrac{\theta}{n}+1)(\dfrac{\theta}{n}+\lambda_{1}+1)\cdots(\dfrac{\theta}{n}+\lambda_{n-1}+1) \text{,}\quad
f\Delta= c\dfrac{\theta}{n}(\dfrac{\theta}{n}+\lambda_{1})\cdots(\dfrac{\theta}{n}+\lambda_{n-1}) \text{,}\quad c>0 $$
and the relations
$$\left[\theta, f\right]= nf \text{,}\qquad\left[\theta, \Delta\right]= -n\Delta \text{.}$$
\newline
Let 
$\mathrm{Mod}^{\mathrm{gr}}(\mathcal{A})$ 
be the category whose objects are finitely generated left
$\mathcal{A}$%
-modules 
$T$
such that for each 
$s \in T$%
, the 
$\mathbb{C}$%
-vector space spanned by the set
$\{\theta^{n}s\; /\; n\geq 1\}$
is finite dimensional.
In other words, this category consists of all graded left
$\mathcal{A}$%
-modules 
$T$
of finite type for 
$\theta$ 
the Euler vector field on
$V$%
.
\medskip

\noindent The functor
$\Psi : \mathrm{Mod}_{\Lambda }^{\mathrm{rh}}(\mathcal{D}_{V})
\longrightarrow \mathrm{Mod}^{\mathrm{gr}}(\mathcal{A})$%
, defined by taking
$\Psi(\mathcal{M})$
to be the set of all
$\mathfrak{g}$%
-invariant
$\theta$%
-homogeneous global sections of
$\mathcal{M}$%
, with quasi-inverse
$\Phi :\mathrm{Mod}^{\mathrm{gr}}(\mathcal{A}) \longrightarrow
\mathrm{Mod}_{\Lambda }^{\mathrm{rh}}(\mathcal{D}_{V})$
defined by
$\Phi(T): = \mathcal{D}_{V}\otimes_{\mathcal{A}} T$%
, give the equivalence of categories:\newline

\noindent \textbf{Theorem \ref{MR}}: Let 
$(G,V)$ 
be a representation of Capelli type with a one-dimensional quotient. Then the categories
$\mathrm{Mod}_{\Lambda }^{\mathrm{rh}}(\mathcal{D}_{V})$
and 
$\mathrm{Mod}^{\mathrm{gr}}(\mathcal{A})$
are equivalent.\bigskip

\noindent This has been conjectured by Levasseur \cite[Conjecture 5. 17, p. 508]{L},
after we had already proved it in the following cases (see \cite{N0}, \cite{N1}, \cite{N2}, \cite{N3}, \cite{N4}, \cite{N5}, \cite{N6}):\medskip 

\begin{itemize}
\item $(G = GL(n)\times{SL(n)},\; V = M_{n}(\mathbb{C}))$
\item $(G = SO(n)\times\mathbb{C}^{*},\; V = \mathbb{C}^{n})$
\item $(G = GL(n),\; V = \Lambda^{2} \mathbb{C}^{n})$%
, 
$n$ 
even
\item $(G = GL(n),\; V = S^{2}\mathbb{C}^{n} )$
\item $(G = Sp(n)\times GL(2), \; \left(\mathbb{C}^{2n}\right)^{2})$\text{.}
\end{itemize}

\noindent Actually, we are interesting to obtain a uniform proof which treated
all the cases at once, that is, the eight cases where 
$(G, V)$
is  of Capelli type with a one-dimensional quotient
(see Appendix A).\newline
The proof of this result is equivalent to the fact that any object in
$\mathrm{Mod}_{\Lambda }^{\mathrm{rh}}(\mathcal{D}_{V})$
is generated by its
$G'$%
-invariant global sections (see Theorem \ref{td}).\newline

It turns out that the equivalence between the categories
$\mathrm{Mod}_{\Lambda }^{\mathrm{rh}}(\mathcal{D}_{V})$
and
$\mathrm{Mod}^{\mathrm{gr}}(\mathcal{A})$
leads to a description of the "analytic" regular holonomic 
$\mathcal{D}_{V}$%
-modules in 
$\mathrm{Mod}_{\Lambda }^{\mathrm{rh}}(\mathcal{D}_{V})$
in terms of "algebraic homogeneous" 
$\mathcal{D}_{V}$%
-modules.
\newline 
 
By the way, we should note that the problem of classifying holomorphic regular holonomic
 $\mathcal{D}$%
 -modules or equivalently perverse sheaves on a complex manifold (thanks to the Riemann-Hilbert correspondence) has been treated by several authors.
The first such result (around 1980) was Deligne's quiver description
of perverse sheaves on an affine line with only possible singularity at the origin \cite{D1}, 
which under the Riemann-Hilbert correspondence is the case where
$G =\mathbb{C}^{\times}$
acts on 
$V=\mathbb{C}$
by scalar multiplication.
Deligne's description uses a characterization of constructible sheaves given in 
\cite{D2}, \cite{D3}. We should also mention the contribution
of  L.
Boutet de Monvel  \cite{BM}, who gave a description of holomorphic regular
holonomic 
$\mathcal{D}$%
-modules in one variable by using pairs of finite dimensional 
$\mathbb{C}$%
-vector spaces and certain linear maps.  A. Galligo,
M. Granger and P. Maisonobe \cite{G-G-M} obtained using the Riemann-Hilbert
correspondence, a classification of regular holonomic 
$\mathcal{D}_{\mathbb{C}^{n}}$%
-modules with singularities along the hypersurface 
$x_{1}\cdots x_{n} = 0$
 by 
 $2^{n}$%
 -tuples of
 $\ \mathbb{C}$%
 -vector spaces with a set of
linear maps. L. Narv\'{a}ez-Macarro \cite{Na} treated the case 
$y^{2}=x^{p}$
using the method of Beilinson and Verdier and generalized this study to the
case of reducible plane curves. R. MacPherson and K. Vilonen \cite{MP-V}
treated the case with singularities along the curve 
$y^{n}=x^{m}$%
. T. Braden and M. Grinberg \cite{B-G} studied perverse sheaves on complex
$n\times{n}$%
-matrices, symmetric matrices and 
$2n\times{2n}$%
-skew-symmetric matrices, each stratified by the rank. They gave an explicit description
of the category of such perverse sheaves as the category of the representations of a quiver.
In \cite{N0}, \cite{N1}, \cite{N5}, the author classified regular holonomic
$\mathcal{D}$%
-modules associated to the same stratification using 
$\mathcal{D}$%
-modules theoretical methods
etc.
This paper is organized as follows:\newline

\noindent In Section \ref{0}, we recall notions on the so called representations of Capelli type.
In section \ref{1}, we review some useful results: in particular the one's saying that: 
any coherent 
$\mathcal{D}_{V}$%
-module equipped with a good filtration, invariant under the action of the Euler vector field
$\theta$%
,
is generated by finitely many global sections of finite type for 
$\theta$%
.
Section \ref{2} deals with the concrete description of 
$\overline{\mathcal{A}}$
the algebra of 
$G'$%
-invariant differential operators following 
Benson - Ratcliff \cite{BR}, Howe - Umeda \cite{H-U1}, Knopp \cite{Kn}, and Levasseur results \cite[Theorem 4.11, p. 491]{L}.
In section \ref{3}, we establish the main result, namely Theorem \ref{MR}. This is done by means of the central Theorem \ref{td}
saying that: any object
$\mathcal{M}$ 
in the category
$\mathrm{Mod}_{\Lambda }^{\mathrm{rh}}(\mathcal{D}_{V})$
is generated by finitely many goblal 
$G'$%
-invariant sections.
This result leads to the equivalence of categories between the category 
$\mathrm{Mod}_{\Lambda}^{\mathrm{rh}}(\mathcal{D}_{V})$ 
and the category 
$\mathrm{Mod}^{\mathrm{gr}}(\mathcal{A})$:
the image by this equivalence of a regular holonomic 
$\mathcal{D}_{V}$%
-module being its set of
$\theta$%
-homogeneous global sections, which are invariant under the action of 
$G'$%
.
\newline
We refer the reader to  \cite{BM1}, \cite{HTT}, \cite{K1}, \cite{K2}, \cite{K3}, \cite{K-K}
for notions on 
$\mathcal{D}$%
-modules theory.

\section{Review on representations of Capelli type with one dimensional quotient}\label{0}

Let
$G$
be a connected reductive complex algebraic group. 
We denote by  
$G'$
its derived subgroup.\newline
Let
$ \rho: G \longrightarrow GL\left(V\right) $
be a finite dimensional representation of
$ G $%
, again denoted by
$ \left(G, V\right) $%
.
Recall that a polynomial 
$ f\in \mathbb{C}\left[V\right] $
is called a relative invariant of
$ \left(G, V\right) $
if there exists a rational character
$\chi \in \mathcal{X}\left(G\right)$
such that
$g\cdot{f}=\chi(g)f$
for all
$ g\in G $%
.
One says 
(see \cite[Chap. 2]{Ki}) that the representation
$ \left(G, V\right) $
is a (reductive) prehomogeneous vector space
if
$ G $
has an open dense orbit
$\Omega$
in
$ V $%
.
In that case, we denote the complement of the open dense orbit by
$ S:= V\backslash\Omega $%
,
it is called the singular set of
$ \left(G, V\right) $%
. 
Then, it is known 
(see \cite[p. 26, theorem 2.9]{Ki})
that, the one-codimensional irreducible components of
$ S $
are of the form
$ \left\{f_{i}=0\right\}, \, 1\leq i\leq r, $
for some relative invariants
$ f_{i} $%
.
The
$ f_{i} $
are algebraically independent, and are called the basic or fundamental relative invariants of
$ \left(G, V\right) $%
.
Note that, any relative invariant can be (up to non zero constant) written as
$ \prod_{i=1}^{r} f_{i}$%
.
When the singular set 
$S$ 
is an hypersurface, the prehomogeneous vector space
$ \left(G, V\right) $
is said to be regular (see \cite[p. 43, theorem 2.28]{Ki}).

\subsection{Multiplicity-free representations}

Let us denote by
$\mathfrak{g}$
the Lie algebra of the connected reductive Lie group
$G$%
,
and by
$\mathfrak{t}$
the Lie algebra of a maximal torus of
$G$%
. 
Denote by
$B$
the set of dominant weights lattices of
$\left(\mathfrak{g}, \mathfrak{t}\right)$%
.
For a fix finite-dimensional representation
$(G,V)$
of the reductive group
$G$%
, we recall that the action of
$G$
on
$V$
extends to the algebra of polynomials on 
$V$%
.
Then, the rational
$G$%
-module
$\mathbb{C}[V]$
decomposes as
\begin{equation}\label{one}
\mathbb{C}[V]\simeq\bigoplus_{\beta \in B}E(\beta)^{m(\beta)}\text{,}
\end{equation}
where
$E(\beta)$
is an irreducible 
$\mathfrak{g}$%
-module with highest weight
$\beta\in B$
and
$m(\beta)\in \mathbb{N}\cup\left\{\infty\right\}$%
.
We recall that the finite-dimensional linear representation
$(G,V)$
is said to be multiplicity-free (MF for short) if its associated representation of
$G$
on
$\mathbb{C}[V]$
decomposes without multiplicities. This means that each irreducible representation
$E(\beta)$
of
$G$
occurs at most once in
$\mathbb{C}[V]$%
. 
More precisely, we recall the following definition \cite[definition 4.1., p. 484]{L}:

\begin{definition}\label{d00}
The representation
$(G,V)$
is called multiplicity-free if in {\rm(\ref{one})}:
$m(\beta)\leq 1$
for all
$\beta$%
.
In this case
$$\mathbb{C}[V]=\bigoplus_{\beta \in B}V(\beta)^{m(\beta)}, \;\; m(\beta)=0, 1,$$
where
$V(\beta)$
is isomorphic to
$E(\beta)$%
.
\end{definition}
Note that, a classification of MF representations can be found in
\cite{BR},\cite{Ka}, \cite{Le}, and a complete list of irreducible MF representations is given in
\cite[table p. 612]{H-U1}
or
\cite[appendix, p. 508]{L}.

\subsubsection{Multiplicity-free spaces with one-dimensional quotient}
As above,
$G'$
is the derived subgroup of the complex Lie group
$G$%
. 
We recall the following definition:
\begin{definition}\label{d2}
{\rm (see Levasseur \cite{L} )} A mutiplicity-free-space
$(G, V)$
is said to have a one-dimensional quotient if there exists a non constant polynomial
$f_{0}\in \mathbb{C}[V]$
such that
$f_{0} \not\in \mathbb{C}[V]^{G}$%
,
and such that
$\mathbb{C}[V]^{G'}= \mathbb{C}[f_{0}]$%
.
\end{definition}

\subsection{Representations of "Capelli type"}

We continue with
$(G,V)$
the finite dimensional representation of the connected reductive Lie group
$G$
. 
We have denoted by 
$\mathfrak{g}=\text{Lie}\left(G\right)$
the Lie algebra of
$G$%
. 
We consider 
$\tau$ 
the differential of the
$G$%
-action
defined as follows:
\begin{equation}\label{diffmap1}
\tau:\mathfrak{g} \longrightarrow \Gamma(V, \mathcal{D})^{\rm{pol}}\text{,}
\end{equation}
where
$\Gamma(V, \mathcal{D})^{\rm{pol}}$
 is the algebra of global algebraic sections of
$\mathcal{D}_{V}$%
, i.e.
the algebra of polynomial coefficients differential operators. For any element 
$\xi$
in
$\mathfrak{g}$%
, the image
$\tau{(\xi)}$
is a linear derivation on
$\mathbb{C}\left[V\right]$
given by
\begin{equation}\label{diffmap2}
\tau(\xi)(\phi)(v)={\frac{d}{dt}}_{|t=0}\left(e^{t\xi}\cdot\phi\right)(v)={\frac{d}{dt}}_{|t=0}\phi{\left(e^{-t\xi}\cdot{v}\right)}\text{,}
\end{equation}
for all
$\phi\in \mathbb{C}[V]$%
,
$v\in V$%
.
This image is homogeneous of degree zero in the sense that
$\left[\theta , \tau{(\xi)}\right] = 0$%
.
Denote by
$U\left(\mathfrak{g}\right)$
the universal enveloping algebra of the Lie algebra
$\mathfrak{g}$%
.
The map
$\tau$
yields a homomorphism denoted again by
$\tau$%
, and defined by
\begin{equation}
\tau: U\left(\mathfrak{g}\right) \longrightarrow \Gamma(V, \mathcal{D}_{V})^{\rm pol}
\text{.}
\end{equation}

\noindent Recall that the group
$G$
acts naturally on
$\Gamma(V, \mathcal{D}_{V})^{\rm pol}$%
:
$\forall\;g \in G\text{,}\;\; \forall\; \phi \in \mathbb{C}[V]\text{,}\;\;\forall\; P\in 
\Gamma\left(V, \mathcal{D}_{V}\right)^{\mathrm{pol}}$%
,
\begin{equation}
 \left(g\cdot{P}\right)(\phi) = g\cdot{P}\left(g^{-1}\cdot{\phi}\right)\text{.}
\end{equation} 
The differential of this action is given by
$P \mapsto \left[\tau(\xi), P\right]$
for
$\xi \in \mathfrak{g},\;  P\in \Gamma(V, \mathcal{D}_{V})^{\rm pol}$%
.
Therefore, a subspace
$I\subset \Gamma(V, \mathcal{D}_{V})^{\rm pol}$ 
is stable under
$G$
(resp.
$G'$)
if and only if
$\left[\tau(\mathfrak{g}), I\right]\subset I$
(resp. 
$\left[\tau(\mathfrak{g}'), I\right]\subset I$).
Then, we know from 
\cite{L} that the subalgebra of polynomial coefficients
$G$%
-invariant differential operators
\begin{equation}
\Gamma(V, \mathcal{D}_{V})^{G}=\left\{ P\in \Gamma(V, \mathcal{D}_{V})^{\rm pol}:\; \left[\tau(\mathfrak{g}), P\right]=0\right\}
\end{equation}
is contained in the one's of
$G'$%
-invariant differential operators
\begin{equation}
\bar{\mathcal{A}}:=\Gamma(V, \mathcal{D}_{V})^{G'}=\left\{ P\in \Gamma(V, \mathcal{D}_{V})^{\rm pol}:\; 
\left[\tau(\mathfrak{g}), P\right]= 0\right\}\text{.}
\end{equation}

\noindent In particular, if
$Z\left(U(\mathfrak{g})\right)=U\left(\mathfrak{g}\right)^{G}$
is the center of
$U\left(\mathfrak{g}\right)$
then
\begin{equation}
\tau\left(Z\left(U(\mathfrak{g})\right)\right) \subset \Gamma\left( V,\mathcal{D}_{V}\right) ^{G}\text{.}
\end{equation}

\noindent Now, we give the following definition 
(see \cite[Definition 5.1.]{L}):

\begin{definition}\label{d3}
We say that the representation
$(G,V)$
is of Capelli type if:

\begin{itemize}
\item $(G,V)$ 
is irreducible and MF;
\item $\tau\left(Z\left(U(\mathfrak{g})\right)\right)=\Gamma \left( V,\mathcal{D}_{V}\right) ^{G}$%
.
\end{itemize}
\end{definition}

\begin{remark} In the list of irreducible MF representations
$\left(G, V\right)$
given by Howe and Umeda
{\rm (see \cite[table p. 612]{H-U1}
or
\cite[appendix p. 508]{L})},
there are exactly eight of them which are of Capelli type with one-dimensional quotient {\rm (see Appendix A)}.
\end{remark}

\section{Coherent 
$\mathcal{D}$%
-modules 
generated by their 
$\theta$%
-homogeneous global sections\label{1}}

We shall denote by 
$\mathcal{D}_{V}$
the sheaf of rings of differential operators on 
$V$ 
with holomorphic coefficients. If 
$x$ 
denotes a typical element of 
$V$%
, and
$\partial := \dfrac{\partial}{\partial{x}}$
its dual in
$\mathcal{D}_{V}$%
,
let
$\theta:= \text{Trace}{(x\partial)}$ 
be the Euler vector field on $V$%
.
\begin{definition}
Let $\mathcal{M}$ be a 
$\mathcal{D}_{V}$%
-module.
A section \
$u$ 
in 
$\mathcal{M}$
is said to be homogeneous if\ $\dim
_{\mathbb{C}}\mathbb{C}\left[ \theta \right] u<\infty $%
, i.e. the 
$\mathbb{C}$%
-vector space spanned by the set
$\{\theta^{n}u\; /\; n\geq 1\}$
is finite dimensional. 
The section 
$u$
is said to be homogeneous of degree 
$\lambda \in \mathbb{C}$%
,  if there exists 
$j\in \mathbb{N}$ 
such that 
$(\theta -\lambda )^{j}u=0$%
.
\end{definition}

\smallskip

\noindent Let us recall the following result which will be used later (see \cite[Theorem 1.3.]{N2}
):\smallskip

\begin{theorem}\label{tu} 
Let 
$\mathcal{M}$ 
be a coherent 
$\mathcal{D}_{V}$%
-module, equipped with a good filtration 
$\left( \mathcal{M}_{k}\right) _{k\in \mathbb{Z}}$
stable under the action of 
$\theta $%
.\ Then,\medskip \smallskip \newline
i) 
$\mathcal{M}$ 
is generated over 
$\mathcal{D}_{V}$ 
by finitely many homogeneous global sections, i.e.,
\begin{equation*}
\mathcal{M} = \mathcal{D}_{V}\left\lbrace s_{1}, \cdots , s_{k}\; \in \Gamma\left( V, \mathcal{M}\right),\; \dim_{\mathbb{C}}\mathbb{C}\left[\theta\right] s_{j} < \infty, \;  0\leq j \leq k\right\rbrace,
\end{equation*}
ii) For any 
$k$ 
$\in \mathbb{N}$%
,
$\lambda \in \mathbb{C}$%
, the vector space 
$\Gamma \left( V,\mathcal{M}_{k}\right) \bigcap \left[
\bigcup\limits_{p\in \mathbb{N}}\ker \left( \theta -\lambda \right) ^{p}%
\right] $ 
of homogeneous global sections in 
$\mathcal{M}_{k}$%
, of degree 
$\lambda$%
, is finite dimensional.
\end{theorem}

\smallskip

\begin{remark}
We will describe a holomorphic classification of regular holonomic $\mathcal{%
D}_{V}$%
-modules in $\mathrm{Mod}_{\Lambda}^{\mathrm{rh}}(\mathcal{D}_{V})$%
, but Theorem \ref{tu} permits to reduce these objects to "algebraic
homogeneous" 
$\mathcal{D}_{V}$%
-modules.
\end{remark}

\section{Algebras of invariant differential operators on a class of mutiplicity-free
spaces\label{2}}
As in the introduction,
$(G,V)$
is a finite-dimensional representation of a connected reductive Lie group
$G$
and 
$G':= [G, G]$
is the derived subgroup of
$G$%
.
Recall that the action of the group
$G$
extends to various algebras, namely
$\mathbb{C}[V]= S(V^{*})$
the algebra of polynomial functions on
$V$%
,
${\Gamma{\left(V, \mathcal{D}_{V}\right)}^{\mathrm{pol}}}$
the algebra of differential operators with polynomial coefficients in
$\mathbb{C}[V]$%
, and
$\mathbb{C}[V^{*}] = S(V)$
identified with differential operators with constant coefficients.
We thus obtain algebras of invariants:
$\mathbb{C}[V]^{G}$%
,
$S(V)^{G}$%
, and
$\Gamma\left(V, \mathcal{D}_{V}\right)^{G}$%
.\newline
If
$(G,V)$
is a prehomogeneous vector space, let 
$f_{0}, \cdots, f_{m}$
be its fundamental relative invariants and let 
$\chi_{j} \in \mathcal{X}(G)$%
,
$0 \leq j \leq m$%
,
be their weight.
There exist relative invariants
$f_{j}^{*}(\partial) \in S(V)$
with weight
$\chi_{j}^{-1}$%
,
$0 \leq j \leq m$ 
(see \cite[Section 3.1]{L}).
We set 
$\Delta_{j}:= f_{j}^{*}(\partial)$
for
$j= 0, \cdots, m$%
.\newline
It is known
that the algebra
$\mathbb{C}[V]^{G'}$
of
$G'$%
-invariant polynomials is a polynomial ring
\begin{equation}
\mathbb{C}[V]^{G'} = \mathbb{C}[f_{0}, \cdots, f_{m}],
\end{equation}
and that
\begin{equation}
S(V)^{G'} = \mathbb{C}[\Delta_{0}, \cdots, \Delta_{m}]
\end{equation}
(see \cite[Lemma 4.2, (d) and formula (4.3) p. 487]{L}).\newline
Consider the following multiplication map
\begin{equation}
\begin{array}{cccc}
 m:& \mathbb{C}[V]\otimes S(V) & \longrightarrow &
\Gamma{\left(V, \mathcal{D}_{V}\right)^{\mathrm{pol}}}\\&&&\\
& \phi\otimes f &\longmapsto & \phi f(\partial)\text{.}
\end{array}
\end{equation}
One knows from
Howe - Umeda
\cite{H-U1}
that through this map the
$\left(\mathbb{C}[V], G\right)$%
-module
$\Gamma\left(V, \mathcal{D}_{V}\right)^{\mathrm{pol}}$
identifies with
$\mathbb{C}[V]\otimes S(V)$:
\begin{equation}
 \Gamma\left(V, \mathcal{D}_{V}\right)^{\mathrm{pol}} \simeq \mathbb{C}[V]\otimes S(V)
\end{equation}
where the group
$G$
acts on 
$\Gamma\left(V, \mathcal{D}_{V}\right)^{\mathrm{pol}}$
as follows:
$ \forall\; \phi \in \mathbb{C}[V],\;\;\forall\; P\in 
\Gamma\left(V, \mathcal{D}_{V}\right)^{\mathrm{pol}}$

\begin{equation}
 \left(g\cdot{P}\right)(\phi) = g\cdot{P}\left(g^{-1}\cdot{\phi}\right)\text{.}
\end{equation}
First, we are interesting in the description of the algebras of
$G$%
-invariant differential operators on a multiplicity-free space following the work by
Benson - Ratcliff \cite{BR}, Howe - Umeda \cite{H-U1}, Knopp \cite{Kn}
and Levasseur \cite{L}.
Actually, the isomorphism 
$m$
is 
$G$%
-invariant,
hence the algebra of
$G$%
-invariant
differential operators decomposes as a direct sum of one-dimensional irreducible
$G$%
-modules
$\mathbb{C}E_{\gamma}$:

\begin{equation}
\Gamma{\left(V, \mathcal{D}_{V}\right)}^{G} = 
\bigoplus_{\gamma\in \Gamma} \mathbb{C}E_{\gamma}
\end{equation}
where
$\Gamma$
is the set of dominant weights lattices of the pair 
$\left(\mathfrak{g}, \mathfrak{t}\right)$  
of the Lie algebras of
$G$
and of a maximal torus of
$G$
respectively.\newline

Let
\begin{equation}
 E_{\gamma}\left(x, \partial_{x}\right) := \dfrac{1}{\dim_{\mathbb{C}}E_{\gamma}}
m\left(E_{\gamma}\right)\qquad
\in
\quad\Gamma{\left(V, \mathcal{D}_{V}\right)}^{G}
\end{equation}
be the operator corresponding to
$E_{\gamma}$%
.
The operators
$E_{\gamma}\left(x, \partial_{x}\right)$
are called the normalized Capelli operators.
Put
\begin{equation}
E_{j}:= E_{\lambda_{j}}\left(x, \partial_{x}\right) \qquad 0\leq j \leq r.
\end{equation}

We know from \cite[Proposition 7.1]{H-U1}
that the given of a multiplicity-free representation
is equivalent to the given of a commutative algebra of
$G$%
-invariant differential operators:
\begin{equation}
(G : V) \quad \text{multplicity-free}\quad \Longleftrightarrow\quad
\Gamma{\left(V, \mathcal{D}_{V}\right)}^{G}\quad
\text{commutative.} 
\end{equation}
In that case the algebra
$\Gamma{\left(V, \mathcal{D}_{V}\right)}^{G}$
is generated by the normalized Capelli operators
$E_{j}$
for
$0\leq j \leq r $
(see \cite[Theorem 9.1]{H-U1} or \cite[Corollary 7.4.4]{BR}):

\begin{theorem}{\rm (Howe - Umeda)}. For a fix multiplicity-free representation
$(G , V)$%
, the algebra
$$\Gamma{\left(V, \mathcal{D}_{V}\right)}^{G} = \mathbb{C}\left[
E_{0}, \cdots, E_{r}\right]$$
is a commutative polynomial ring.
\end{theorem}

From now on, we focus our attention in the subalgebras of 
$G$
(resp. 
$G'$%
)-invariant global algebraic sections
of
$\mathcal{D}_{V}$
on
multiplicity-free representations with a one-dimensional quotient.

\subsection{Invariant differential operators on multiplicity-free spaces with one dimensional
quotient}

Recall that
$G'$
denotes the derived subgroup of
$G$%
.
Recall also that a multplicity-free representation
$(G , V)$
is said to be with one-dimensional quotient
if there exists a polynomial function
$f \in \mathbb{C}[V]$
such that
\begin{equation}
 {\mathbb{C}[V]}^{G'}= \mathbb{C}[f]\qquad \text{and} \qquad  f\;\not\in \; {\mathbb{C}[V]}^{G}\text{.}
\end{equation}
In fact, the polynomial function
$f$
is a relative invariant of degree
$n$
of weight
$\chi \in \mathcal{X}(G)$%
,
and there exists an associated relative invariant differential operator
$f^{*} := f(\partial)\; \in\; \mathbb{C}[V^{*}]$
of degree 
$n$
with weight
$\chi^{-1}$%
.
More precisely, set
$\Delta := f^{*}(\partial)$%
.
We know from Sato - Bernstein - Kashiwara (see \cite[Proposition 2.22]{Ki} and \cite{K0}) that there 
exists a polynomial
$b(s) \in \mathbb{R}[s]$
of degree 
$n$
called the Bernstein - Sato polynomial such that:

\begin{equation}
\begin{array}{ccccc}
i) & b(s) & = & c\prod _{j=0}^{n-1}(s + \lambda_{j} + 1), &  c > 0; \cr  \\
 \\
ii) &\Delta{(f^{s+1})} & = & b(s)f^{s}; \qquad\qquad & \qquad \qquad \cr\qquad \\
\\
iii) & \lambda_{j+1}\in & \mathbb{Q}^{*}{}^{+}, & 0\leq j \leq n-1,&  \lambda_{0} = 0\text{.}  \cr\\
\end{array}
\end{equation}
Set
\begin{equation}
f := f_{0} \quad
\text{and}\quad
\Delta := \Delta_{0} = f^{*}(\partial)
\text{.}
\end{equation}
Following T. Levasseur \cite[Section 4.2]{L}, recall that if
$(G,V)$
is a multiplicity-free representation of one-dimensional quotient
then we have
\begin{equation}
\mathbb{C}[V]^{G'} = \mathbb{C}[f],\qquad S(V)^{G'}=\mathbb{C}[V^{*}]^{G'} = \mathbb{C}[\Delta]\qquad
\text{and} \qquad E_{0} = f\Delta \text{.}
\end{equation}
Now, consider
$\overline{\mathcal{A}}: = \Gamma\left( V, \mathcal{D}_{V}\right)^{G'}$
the algebra of
$G'$%
-invariant
(polynomial coefficients) differential operators on 
$V$:
\begin{equation} 
\overline{\mathcal{A}} \; \supset \; \Gamma\left( V, \mathcal{D}_{V}\right)^{G}\quad
\text{and}\quad
\mathcal{J}:= \left\lbrace P \in \Gamma\left( V, \mathcal{D}_{V}\right)^{G}/ \quad Pf^{m}= 0\quad
\text{for all}\;\; m\in \mathbb{N} \right\rbrace \subset \overline{\mathcal{A}}
\end{equation}
is the annihilator of the 
$G'$%
-invariant polynomial functions on
$V$%
.\newline
Recall that 
$\theta$
denotes the Euler vector field on
$V$%
,
$\theta \in \Gamma\left( V, \mathcal{D}_{V}\right)^{G}$%
.
T. Levasseur
\cite[Lemma 4.10]{L}
proved that: for any 
$G$%
-invariant differential operator
$P\in \Gamma\left( V, \mathcal{D}_{V}\right)^{G}$
, there exists an associated Bernstein-Sato polynomial
$b_{P}(s)\in \mathbb{C}[s]$
such that the operator
$P - b_{P}(\theta)$
belongs to
$\mathcal{J}$%
.
In particular, one can find a polynomial
$b_{E_{j}}(s)$
associated with each Capelli operator
$E_{j}$%
, 
$0\leq j\leq r$%
, such that if we consider
$\Omega_{j}$
to be
\begin{equation}
\Omega_{j}:= E_j - b_{E_{j}}(\theta)\;\in\; \mathcal{J}\quad \text{for}\;\; j= 0, \cdots ,r,
\end{equation}
then we obtain the following results \cite[Theorem 4.11, (i), (v)]{L}:

\begin{theorem} If 
$(G,V)$
is a fix multplicity-free representation with one-dimensional quotient,
then
\begin{equation}
 \overline{\mathcal{A}} =  \mathbb{C}\left\langle{f, \Delta, \theta, \Omega_{1}, \cdots ,\Omega_{r} }\right\rangle ,
\end{equation}
\end{theorem}

\begin{equation}
\mathcal{J}  =  \Sigma_{j=1}^{r}\overline{\mathcal{A}}\Omega_{j}\text{.}
\end{equation}
Note that, the operators
$f$
and
$\Delta$
do not commute nor do not commute with the operators
$\Omega_{1}, \cdots ,\Omega_{r}$%
.\newline
By the way, using these results, T. Levasseur
\cite[Theorem 4.15]{L}
gives a duality (of Howe type) correspondence between (multplicity-free) 
representations (with a one-dimensional quotient) of
$G$
and lowest weight modules over the Lie algebra generated by
$f$
and
$\Delta$
(which is infinite dimensional when the degree of
$f$ is
$\geq 3$).
Actually, this duality recovers and extends results obtained by H. Rubenthaler
when the representation
$(G, V)$
is of "commutative parabolic type"
(see \cite[Proposition 4.2]{R0}
and also
\cite[Corollary 4.5.17]{G-W}).\newline
We should note that when 
$(G, V)$
is irreducible, then
\begin{equation}
\Omega_{r} = 0\text{,} \quad
\text{the two sided ideal}\quad
\mathcal{J}= \Sigma_{j=0}^{r-1}\overline{\mathcal{A}}\Omega_{j} = 
\Sigma_{j=0}^{r-1}\Omega_{j}\overline{\mathcal{A}}
\text{,\quad and}
\end{equation}

\begin{equation}
 \overline{\mathcal{A}} =  \mathbb{C}\left\langle{f, \Delta, \theta, \Omega_{1}, \cdots ,\Omega_{r-1} }\right\rangle\text{.}
\label{Result}
\end{equation}
In the case
$(GL(n, \mathbb{R}), S^{2}(\mathbb{R}^{n}))$
of the real general linear group action on real symmetric matrices, M. Muro proved this formula in
\cite[Proposition 2.1, p. 356]{Mu}.
When
$(G, V) = (GL(n, \mathbb{C})\times SL(n, \mathbb{C}), M_{n}(\mathbb{C}))$%
,
$(GL(2m, \mathbb{C}),\; \Lambda^{2} \mathbb{C}^{2m})$%
,
$(GL(n, \mathbb{C}),\; S^{2}\mathbb{C}^{n} )$%
, this non commutative algebra is obtained with explicit relations in
\cite[Proposition 5, p. 637-638 ]{N1}, \cite[Proposition 6, p. 120 ]{N0}, \cite[Proposition 8, p. 4]{N5}.
 Actually,
the result (\ref{Result}) generalizes the one's of H. Rubenthanler
(see \cite[Proposition 3.1]{R1} or
\cite[Theorem 5.3.3.]{R2})
obtained when
$(G, V)$
is an irreducible regular prehomogeneous representation of commutative parabolic type.
We have the following proposition.

\begin{proposition}\label{pu} Let
$(G,V)$
be an irreducible multiplicity-free representation with a one-dimensional quotient.
The following relations hold in the quotient algebra
$\overline{\mathcal{A}}/\mathcal{J}$:\newline
\begin{eqnarray}
\left[ \theta, f \right] & = & nf \text{,}\label{a0}
\\
\left[\theta, \Delta \right] & = & -n\Delta  \text{,}\label{b0}
\\
f\Delta & = & c\dfrac{\theta}{n}(\dfrac{\theta}{n}+\lambda_{1})\cdots(\dfrac{\theta}{n}+\lambda_{n-1}) \text{,}\quad c>0 \label{c0}
\\
\Delta{f} & = & c(\dfrac{\theta}{n}+1)(\dfrac{\theta}{n}+\lambda_{1}+1)\cdots(\dfrac{\theta}{n}+\lambda_{n-1}+1) \text{,} \label{d0}
\\
f_{j}\Delta_{j} & = & c_{j}\dfrac{\theta}{n}(\dfrac{\theta}{n}+\lambda_{1})\cdots(\dfrac{\theta}{n}+\lambda_{n-j-1}),\;\;
c_{j} > 0,\quad 0\leq j\leq r \label{e0}
\end{eqnarray}
where 
$\lambda_{k} \in \mathbb{Q}$
for 
$k = 0,\cdots, n-1$
\end{proposition}

\begin{proof}
We should note that by \cite[Remark 4.12, (2)]{L},
we have the homogeneity of degree
$n$
(resp.
$-n$)
of the polynomial
$f$
(resp.
$\Delta$%
),
that is,
the formula
(\ref{a0}),
(\ref{b0}).\newline
Recall that
$\Omega_{j}:= E_j - b_{E_{j}}(\theta)\;\in\; \mathcal{J}$%
, 
for\; 
$j= 0, \cdots ,r$%
,
so we clearly have
\begin{equation}
E_j = b_{E_{j}}(\theta)\;\;
\text{in}\;\;
\overline{\mathcal{A}}/\mathcal{J}\label{clear}.
\end{equation}
Recall also that from
\cite[p. 490]{L},
we have
$E_{0}= f\Delta$
and 
$b_{E_{0}}(s)= b(s-1)$
where
$b(s)= c(s+1)(s+\lambda_{1} + 1)\cdots (s + \lambda_{n-1} + 1)$
is the 
$b$%
-function of
$f$%
. 
Then, using this last in (\ref{clear}), we get (\ref{c0})
$$ f\Delta  =  c\dfrac{\theta}{n}(\dfrac{\theta}{n}+\lambda_{1})\cdots (\dfrac{\theta}{n}+\lambda_{n-1}) \quad \text{in} \quad 
\overline{\mathcal{A}}/\mathcal{J}\text{.}$$
Next, since
$\Delta f^{s+1} = b(s)f^{s}$%
, that is,
$(\Delta{f}) f^{s} = b(s)f^{s}$
we get the formula (\ref{d0}):
$$\Delta{f} = b(\theta)\mod  \mathcal{J}\text{.}$$
More generally, we may take
$E_{j} = f_{j}\Delta_{j}$
and using (\ref{clear}) we get
$$
f_{j}\Delta_{j} = b_{E_{j}}(\theta)\;\;
\text{in}\;\;
\overline{\mathcal{A}}/\mathcal{J}
$$
with
$b_{E_{j}}(s) = b_{j}(s-1) = c_{j}s(s+\lambda_{1})\cdots (s + \lambda_{n-j-1})$%
,
$\;c_{j} > 0$%
, 
$\;0\leq j\leq r$%
, that is, the formula
(\ref{e0}).
\end{proof}\newline

Let 
$\mathcal{K}$
be the ideal of
$\overline{\mathcal{A}}/\mathcal{J}$
defined by the relations 
(\ref{a0}), (\ref{b0}), (\ref{c0}), (\ref{d0})
of Proposition \ref{pu}.
Then the preimage of
$\mathcal{K}$
under the quotient map
$\overline{\mathcal{A}} \longrightarrow \overline{\mathcal{A}}/\mathcal{J}$ 
is an ideal of
$\overline{\mathcal{A}}$
containing properly
$\mathcal{J}$%
.
Let us denote by 
$\overline{\mathcal{J}}$
the preimage in
$\overline{\mathcal{A}}$ 
of the ideal
$\mathcal{K}$%
. 
Denote by 
$\mathcal{A}$ 
the quotient algebra of 
$\overline{\mathcal{A}}$
by 
$\overline{\mathcal{J}}$:
\begin{equation}
\mathcal{A}:=\overline{\mathcal{A}}/\overline{\mathcal{J}}\text{.}
\end{equation}
We have the following corollary
which is a particular case of T. Levasseur's result in
\cite[Theorem 3.9, p. 483]{L} or H. Rubenthaler
\cite[Theorem 2.8, p. 1345]{R1}, \cite[Theorem 7.3.2, p. 37]{R2}:

\begin{corollary}
\label{cu}The quotient algebra $\mathcal{A}$ is generated by $f ,\theta, 
\Delta$ 
satisfying the relations (\ref{a0}), (\ref{b0}), (\ref{c0}), (\ref{d0}): 
\begin{eqnarray*}
\left[ \theta, f \right] & = & nf \text{,} \\
\left[\theta, \Delta \right] & = & -n \text{,}\Delta\\
f\Delta & = &  c\dfrac{\theta}{n}(\dfrac{\theta}{n}+\lambda_{1})\cdots(\dfrac{\theta}{n}+\lambda_{n-1}) \text{,} \\
\Delta{f} & = & c(\dfrac{\theta}{n}+1)(\dfrac{\theta}{n}+\lambda_{1}+1)\cdots(\dfrac{\theta}{n}+\lambda_{n-1} +1)\text{.} 
\end{eqnarray*}
\end{corollary}

\section{ 
$\mathcal{D}_{V}$%
-modules on representations of "Capelli type" with one-dimensional quotient
generated by their invariant global sections\label{3}}

In this section, we continue with the representation 
$(G, V)$
of the connected (reductive) Lie group
$G$
as in Section \ref{2},
and
$G'$
its derived subgroup.
It is well known, in this case, that
$G$
(resp.
$G'$%
) acts on 
$V$
with finitely many orbits
$(V_{k})_{k\in K}$
(see \cite{Ka}).
Let
$\Lambda \subset T^{\ast}V$
be the Lagrangian subvariety which is the union of the closure of conormal bundles
${T_{V_{k}}^{\ast }V}$%
, where
$V_{k}$
are the orbits of
$G$
(see Panyushev \cite{Pa}).
We recall that the action of
$G$
on
$V$
defines a morphism (see 
(\ref{diffmap1}),
(\ref{diffmap2}))
$\tau: \mathfrak{g} \longrightarrow \Theta_{V}, \; \xi \mapsto \tau(\xi)$
from the Lie algebra
$\mathfrak{g}$
of
$G$
to the subalgebra
$\Theta_{V}$
of
$\mathcal{D}_{V}$
consisting of vector fields on 
$V$%
, i.e. the tangent sheaf on
$V$%
.
So the Lagrangian variety
$\Lambda$ 
is defined by the common zeros of the principal symbols of vector
fields corresponding to infinitesimal generators of
$G$%
.\newline
Recall that a
$\mathcal{D}_{V}$%
-module is said to be holonomic if it is coherent and its characteristic variety is Lagrangian.
Equivalently the characteristic variety is of dimension equal to 
$\dim{V}$%
.
A holonomic 
$\mathcal{D}_{V}$%
-module
$\mathcal{M}$
is regular if there exists a global good filtration
$F\mathcal{M}$
on
$\mathcal{M}$
such that the annihilator of 
$\mathrm{gr}^{F}\mathcal{M}$
(i.e., the ideal
$\mathrm{ann}_{\mathbb{C}[T^{*}V]}\mathrm{gr}^{F}\mathcal{M}$%
) 
is a radical ideal in
$\mathrm{gr}^{F}\mathcal{M}$
(see \cite[definition 5.2]{K1} or 
\cite[Corollary 5.1.11]{KK1}).
As in the introduction, we denote by 
$\mathrm{Mod}_{\Lambda }^{\mathrm{rh}}(\mathcal{D}_{V})$
the full category consisting of all holomorphic regular holonomic 
$\mathcal{D}_{V}$%
-modules whose characteristic variety
$\Lambda$
is contained in the union of the closure of conormal bundles to the 
$G$%
-orbits 
(see Panyushev \cite{Pa}).
Let
$\mathcal{M}$
be a holomorphic regular holonomic
$\mathcal{D}_{V}$%
-module in
$\mathrm{Mod}_{\Lambda }^{\mathrm{rh}}(\mathcal{D}_{V})$%
. We know from Brylinski and Kashiwara
\cite[p. 389, (1.2.4)]{BK} that
$\mathcal{M}$
has a good filtration 
$(\mathcal{M}_{j})_{j\in \mathbb{Z}}$
satisfying the following condition:\newline
For a differential operator
$P$
of degree
$m$
(%
$P\in \Gamma (U, \mathcal{D}_{V}(m))$%
, where 
$U$
is an open subset of
$V$%
),
if its principal symbol
$\sigma_{m}(P)$
vanishes on the characteristic variety
${\rm char}(\mathcal{M})$%
, then we have
\begin{equation}
P\mathcal{M}_{j} \subset \mathcal{M}_{j+m-1} \quad \text{for any}\; j\in \mathbb{Z} \text{.}\label{action}
\end{equation}
In particular, if
$\xi$
is a vector field
(corresponding to an infinitesimal generator of
$G$%
) which describes the characteristic variety 
$\Lambda$%
, its principal symbol vanishes on
$\Lambda \supset {\rm char}(\mathcal{M})$
(so vanishes on
${\rm char}(\mathcal{M})$%
). Then the relation
(\ref{action}) implies that
\begin{equation}
\xi\mathcal{M}_{j} \subset \mathcal{M}_{j+1-1}\text{,}  \quad
\text{that is}
\end{equation}
\begin{equation}
\xi\mathcal{M}_{j} \subset \mathcal{M}_{j} \quad \text{for any}\; j\in \mathbb{Z} \text{.}
\end{equation}
Then we have the following
\begin{remark}\label{special}
The objects of the category 
$\mathrm{Mod}_{\Lambda }^{\mathrm{rh}}(\mathcal{D}_{V})$
are holomorphic regular holonomic
$\mathcal{D}_{V}$%
-modules equipped with global good filtrations
which are preserved by the action of the Lie algebra
$\mathfrak{g}$
of
$G$%
.
\end{remark}

We recall the folloing definition:
\begin{definition}\label{du}
Let
$G$
be an algebraic group acting on a smooth variety
$V$%
, and 
$\alpha :G\times V\longrightarrow V$
the group action morphism 
$\left(\alpha\left( g,v\right) = g\cdot v \; \left( g\in G,\; v\in V\right)\right)$%
.
One says that the group
$G$ 
acts on a well filtered
$\mathcal{D}_{V}$%
-module 
$\mathcal{M}$ 
if it preserves the good filtration on
$\mathcal{M}$%
, and there exists an isomorphism of 
$\mathcal{D}_{G\times V}$%
-modules 
$u:\alpha^{+}(\mathcal{M})\overset{\sim }{\longrightarrow }{\rm pr}_{V}^{+}(\mathcal{M})$
satisfying the associativity condition coming from the group multiplication of
$G$
$\left(
{\rm pr}_{V}:G\times V\longrightarrow V,\left( g,v\right) \longmapsto v\;
\text{is the projection onto}\;
V 
\right)$%
.
\end{definition}

We specialize further to the case where
$(G, V)$
is of Capelli type, i.e., 
$(G, V)$
is an irreducible multiplicity-free-space such that
$\Gamma\left(V, \mathcal{D}_{V}\right)^{G}$
is equal to the image of the center of
$U(\mathfrak{g})$
under the differential 
$\tau : \mathfrak{g} \longrightarrow \Gamma\left(V, \mathcal{D}_{V}\right)^{\rm pol}$
of the 
$G$%
-action (see definition \ref{d3}). More precisely,  assume that
$(G,V)$
is a representation of Capelli type with a one-dimensional quotient, i.e., there exists a non constant polynomial
$f$
such that
$f \not\in \mathbb{C}[V]^{G}$%
, and such that
$\mathbb{C}[V]^{G'}= \mathbb{C}[f]$
(see definition \ref{d2}).
Let
$G_{1}$
be the simply connected cover of the derived subgroup
$G'$%
.
T. Levasseur
\cite[Lemma 5.15]{L}
proved that the category of
$(G_{1}\times{C})$%
-equivariant
$\mathcal{D}_{V}$%
-modules, where
$C$
is the centre of
$G$%
, is equivalent to the category 
$\mathrm{Mod}_{\Lambda }^{\mathrm{rh}}(\mathcal{D}_{V})$
of holomorphic regular holonomic
$\mathcal{D}_{V}$%
-modules studied here.
Therefore, we deduced the following remark:
\begin{remark}
\label{pz}The action of
$G$ 
on 
$V$
extends to an action of the universal covering 
$G_{1}$ 
on 
$ \mathcal{D}_{V}$%
-modules
$\mathcal{M}$ 
in 
$\mathrm{Mod}_{\Lambda }^{\mathrm{rh}}\left( \mathcal{D}_{V}\right)$%
.
Specially the derived subgroup
$G'$
acts on
$\mathcal{M}$%
.
\end{remark}

This section consists in the proof of the main general argument of the paper. We
show that any 
$\mathcal{D}_{V}$%
-module 
$\mathcal{M}$
in the category 
$\mathrm{Mod}_{\Lambda }^{\mathrm{rh}}(\mathcal{D}_{V})$ 
is generated by its invariant global sections under the action of 
$G'$%
.
\begin{theorem}
\label{td}A $\mathcal{D}_{V}$%
-module 
$\mathcal{M}$ 
in 
$\mathrm{Mod}_{\Lambda}^{\mathrm{rh}}(\mathcal{D}_{V})$ 
is generated by its 
$G'$%
-invariant
global sections.
\end{theorem}

Firstly, we give some basic results which will be used in the proof of this central theorem.

\subsection{Extension of sections and 
$G$%
-invariance}

\noindent For the proof of Theorem \ref{td}, we shall use an algebraic point of view. 
Since the concerning
$\mathcal{D}_{V}$%
-modules are regular holonomic, it is equivalent to consider the algebraic case or the analytic one. We need the following two lemmas in the proof:\newline
\begin{lemma}\label{Serre}
{\rm(\cite[Lemma 1, p. 247, $n^{\circ} 55$]{S})}\newline
Let 
$V$
be an affine variety,
$f$
a regular function on
$V$%
, and
$\Omega$
the set of points
$x\in V$
such that 
$f(x)\neq 0$%
. Let 
$\mathcal{F}$
be a coherent algebraic sheaf on
$V$%
, and
$s \in \Gamma\left(\Omega, \mathcal{F}\right)$
a section of 
$\mathcal{F}$
on 
$\Omega$%
.
Then, for any large enough
$N \in \mathbb{N}$%
, there exists a section 
$s'$
of
$\mathcal{F}$
on the whole
$V$
{\rm{(}}%
$s' \in \Gamma\left(V, \mathcal{F}\right)$%
{\rm{)}}%
, such that 
$s' = sf^{N}$
on
$\Omega$%
, i.e.,
\begin{equation}
s'_{|_{\Omega}} = sf^{N} \text{.}
\end{equation}
\end{lemma}

\begin{lemma}\label{Nitin}
Consider
$G'$
the complex algebraic group acting on the affine algebraic variety
$V$%
,
$f$
a
$G'$%
-invariant regular function on
$V$
$\left(f\in \mathbb{C}[V]^{G'}\right)$%
,
$\Omega$
the complement in
$V$ 
of the hypersurface defined by 
$f=0$%
, and 
$\mathcal{F}$
a 
$G'$%
-equivariant coherent algebraic sheaf on
$V$%
.
Then, any 
$G'$%
-invariant section 
$s$
of
$\mathcal{F}$
on
$\Omega$
$\left( s \;\in \;\Gamma\left(\Omega, \mathcal{F}\right)^{G'}\right)$
extends to a 
$G'$%
-invariant global section
$m$
$\left( m \;\in \;\Gamma\left(V, \mathcal{F}\right)^{G'}\right)$%
.
\end{lemma}
\begin{proof}
 Recall that
$V$
is an affine algebraic variety, i.e. 
$V = \text{Spec} A$%
, where
$A := \mathbb{C}[V]$
is an affine algebra over 
$\mathbb{C}$
and
$\Omega = \text{Spec} A[\frac{1}{f}]$
with
$A[\frac{1}{f}] = \mathbb{C}[V][\frac{1}{f}]= \mathbb{C}[\Omega]$%
.\newline
Since 
$\mathcal{F}$
is a coherent algebraic sheaf on
$V$%
, then
$\mathcal{F}$
is a finitely generated 
$A$%
-module. We consider the restriction of
$\mathcal{F}$
on
$\Omega$%
:
\begin{equation}
\mathcal{F}[\Omega] : =\mathcal{F}\bigotimes\limits_{A} A[\frac{1}{f}]\text{.}\label{*}
\end{equation}
The previous lemma says that any section
$s$
of
$\mathcal{F}$
on
$\Omega$
$\left( s \;\in \;\Gamma\left(\Omega, \mathcal{F}\right)\right)$
extends to a global section
$m$
$\left( m \;\in \;\Gamma\left(V, \mathcal{F}\right)\right)$
such that
\begin{equation}
 m_{|_{\Omega}} = sf^{p} \quad \text{for} \;\;  p \gg 0 \text{.} \label{**}
\end{equation}
So, from (\ref{*})
and
(\ref{**}),
the section 
$s$
can be written as
\begin{equation}
 s = \dfrac{m}{f^{r}} \quad \text{for} \;\;r \gg 0\text{.}\label{***}
\end{equation}
Recall that the group
$G'$
acts on
$A$
and on
$\mathcal{F}$%
. Then, for any 
$g \in G'$
acting on
$s$%
, we have
\begin{equation}
 g.s = g.\left(\dfrac{m}{f^{r}}\right) = \dfrac{g.m}{g.f^{r}} \text{.} \label{****}
\end{equation}
Since
$s$
 is  a 
$G'$%
-invariant section
(%
$g.s = s$%
)
and
$f$
is a 
$G'$%
-invariant regular function
(%
$f = g.f$%
), then the previous equality becomes:
\begin{equation}
 s = \dfrac{g.m}{f^{r}}\text{.}
\end{equation}
Using
(\ref{***})
we get
\begin{equation}
 \dfrac{m}{f^{r}} = \dfrac{g.m}{f^{r}}  \quad \Longleftrightarrow \quad \dfrac{m-g.m}{f^{r}} = 0\text{.}
\end{equation}
This means that there exists a large integer 
$N \gg 0$
such that 
\begin{equation}
 (m - g.m) f^{N} = 0 \quad \Longleftrightarrow \quad mf^{N} = (g.m)f^{N}\text{.}
\end{equation}
Since 
$f$
is 
$G'$%
-invariant (%
$f^{N} = g.f^{N}$%
), 
this last becomes
\begin{equation}
mf^{N} = (g.m)(g.f^{N})\text{,}
\end{equation}
that is,
\begin{equation}
 mf^{N} = g. (mf^{N})\text{.}
\end{equation}
Thus 
$mf^{N}$
is a 
$G'$%
-invariant global section extending
$s$ 
$\left( mf^{N} \;\in \;\Gamma\left(V, \mathcal{F}\right)^{G}\right)$
\end{proof}

\subsection{Proof of theorem \ref{td}}
Recall that the irreducible multiplicity free representation
$(G,V)$
has a Zariski open dense orbit
$\Omega$%
,
and a relative invariant
$f$
(i.e., there exists a character
$\chi \in \mathcal{X}{(G)}$
such that
$g\cdot{f} = \chi{(g)} f$
for
$g\in G$%
)
which is a
$G'$%
-invariant homogeneous polynomial of degree
$n$
such that
$\mathbb{C}[V]^{G'} = \mathbb{C}[f]$%
. In this case, we know from V. G. Kac \cite{Ka}
that 
$G$
has finitely many orbits, namely
$n+1$
orbits. We denote by
$\overline{V_{k}}$
the closure of the
$G$%
-orbits
$V_{k}$
for
$0 \leq k \leq n$
with
$V_{0} = \{0\}$%
. Let us consider again
$f$
as the mapping
$f: V \longrightarrow \mathbb{C}, \; x \mapsto f(x)$%
, and 
$\overline{V}_{n-1}$
the hypersurface defined by 
$f = 0$%
, then we have
$\Omega  := V\backslash \overline{V}_{n-1}$
the complement in
$V$
of the 
$\overline{V}_{n-1}$%
.\newline
 
Let 
$\mathcal{M}$
be a holomorphic regular holonomic
$\mathcal{D}_{V}$%
-module in the category
$\mathrm{Mod}_{\Lambda }^{\mathrm{rh}}\left( \mathcal{D}_{V}\right)$%
. One sets
$$\mathcal{M}^{G'}:=\mathcal{D}_{V}\lbrace m_{1}, \cdots , m_{p} \; \in \;\Gamma(V,\mathcal{M})^{G'}\;\; \text{such that}\;\;\mathrm{dim}_{\mathbb{C}}\mathbb{C}[\theta]m_{j} <\infty \;\;\text{for}\; 1\leq j\leq p\rbrace$$
the submodule of 
$\mathcal{M}$
generated, over
$\mathcal{D}_{V}$%
, by finitely many homogeneous global sections, which are invariant under the action of
$G'$%
.

First, we claim that on the open dense orbit
$\Omega$%
,
we have the equality
$\mathcal{M} = \mathcal{M}^{G'} $%
. \newline 
\smallskip
Indeed, let 
$j: \Omega \longrightarrow V$
be the open embedding. The restriction
$\mathcal{M}_{\Omega} := j^{+}\left(\mathcal{M}\right)$
is a 
$G'$%
-equivariant
$\mathcal{D}_{\Omega}$%
-module. Notice that, if we denote again by 
$f$
the mapping
$f : V \longrightarrow \mathbb{A}^{1}$
, this identifies 
$\Omega / G$
with
$\mathbb{G}_{m} = \mathbb{A}^{1}\backslash\lbrace 0\rbrace$%
. The generic stabilizers 
$H$ 
in 
$G'$
of points in 
$\Omega$
are connected (see Appendix C, Remark), so the 
$G'$%
-equivariant
$\mathcal{D}_{\Omega}$%
-module
$\mathcal{M}_{\Omega}$
is the pullback by 
$f$
of a 
$\mathcal{D}_{\Omega / G}$%
-module 
$\mathcal{N}$
on
$\Omega / G$%
:
\begin{equation}
\mathcal{M}_{\Omega} = f^{+}\left(\mathcal{N}\right) \quad \text{with}\quad  \mathcal{N}\; \text{a}\;
\mathcal{D}_{\Omega / G}%
\text{-module}\text{.}
\end{equation}
Thus on 
$\Omega$%
, the
$G'$%
-invariant sections of
$\mathcal{M}_{\Omega}$%
, i.e., 
$\Gamma\left( \Omega , \mathcal{M}_{\Omega}\right)^{G'}$
(which are exactly the inverse images by 
$f$
of
$\Gamma\left(\mathbb{G}_{m}, \mathcal{N}\right)$ 
the sections on
$\mathbb{G}_{m}$
of
$\mathcal{N}$%
) generate
$\Gamma\left( \Omega , \mathcal{M}_{\Omega}\right)$
as a
$\Gamma\left(\Omega, \mathcal{O}_{\Omega}\right)$%
-module:
\begin{equation}
\Gamma\left( \Omega , \mathcal{M}_{\Omega}\right)^{G'} = f^{-1}\left(\Gamma\left(\mathbb{G}_{m}, \mathcal{N}\right)\right)\text{,}
\end{equation}
and
\begin{equation}\label{star}
\Gamma\left( \Omega , \mathcal{M}_{\Omega}\right) = \Gamma\left(\Omega, \mathcal{O}_{\Omega}\right) 
\left\lbrace \Gamma\left( \Omega , \mathcal{M}_{\Omega}\right)^{G'} \right\rbrace
= \Gamma\left(\Omega, \mathcal{O}_{\Omega}\right)
\left\lbrace f^{-1}\left( \Gamma\left(\mathbb{G}_{m}, \mathcal{N}\right)\right) \right\rbrace\text{.}
\end{equation}
Now, for every section
$m \in \Gamma\left( \Omega , \mathcal{M}_{\Omega}\right)$%
,
one can find a sufficiently large integer
$N \gg 0$
such that the section obtained by multiplication by
$f^{N}$%
, that is,
\begin{equation}
mf^{N}\; \in \; \Gamma\left( \Omega , \mathcal{M}_{\Omega}\right)
\end{equation}
extends to a global section of
$\mathcal{M}$
(see Lemma \ref{Serre}), i.e., the section
$mf^{N}$
lifts to a global section
\begin{equation} 
\widetilde{mf^{N}}
\; \in \;
\Gamma\left( V, \mathcal{M}\right)
\text{.}
\end{equation}
If 
$m$
is a 
$G'$%
-invariant section on
$\Omega$
(%
$m \in \Gamma\left( \Omega , \mathcal{M}_{\Omega}\right)^{G'}$%
), so is 
$mf^{N}$%
, i.e., 
\begin{equation}
mf^{N}\in \Gamma\left( \Omega , \mathcal{M}_{\Omega}\right)^{G'}\text{.}
\end{equation}
Then, according to the Lemma \ref{Nitin}, we can choose this lifting section
$\widetilde{mf^{N}}$
to be 
$G'$%
-invariant:
\begin{equation} 
\widetilde{mf^{N}}
\; \in \;
\Gamma\left( V, \mathcal{M}\right)^{G'}
\text{.}
\end{equation}
Thus, by
(\ref{star})
(and since the mapping
$f$
is invertible on
$\Omega$%
), the image of
$\Gamma\left( V , \mathcal{M}\right)^{G'}$
in
$\Gamma\left( \Omega , \mathcal{M}_{\Omega}\right)^{G'}$
generates 
$\Gamma\left( \Omega , \mathcal{M}_{\Omega}\right)^{G'}$
as a
$\Gamma\left(\Omega, \mathcal{O}_{\Omega}\right)$%
-module.

\noindent Since
$\Omega$ 
is an affine space, we see that the restriction of
$\mathcal{M}^{G'}$
to
$\Omega$
equals
$\mathcal{M}_{\Omega}$%
:
\begin{equation}
j^{+}\left(\mathcal{M}^{G'}\right) = \mathcal{M}_{\Omega}\text{.}
\end{equation}
Hence on 
$\Omega$%
, the quotient module
$\mathcal{M}/ \mathcal{M}^{G'}$
is zero, namely
\begin{equation}
\mathcal{M}/ \mathcal{M}^{G'} = 0 \quad \text{on} \quad\Omega \text{,}
\end{equation}
and its support lies in the hypersurface
$\overline{V}_{n-1}$%
:
\begin{equation}
\rm{Supp}\left(\mathcal{M}/ \mathcal{M}^{G'}\right) \subset\; \overline{V}_{n-1}\text{.}
\end{equation}\newline

Now, since we already know that 
$\mathcal{M}$
is a 
$G'$%
-equivariant
$\mathcal{D}_{V}$%
-module (see Remark \ref{pz}), then 
$\mathcal{M}^{G'}$
is also
$G'$%
-equivariant, hence such is the quotient module
$\mathcal{M}/ \mathcal{M}^{G'}$%
. Moreover, since 
$\overline{V}_{n-1}$
has a finite number of 
$G'$%
- orbits, we have
$\mathcal{M}/ \mathcal{M}^{G'}$ 
with support on the closure of the
$G'$%
-orbits, i.e.,
\begin{equation}
\rm{Supp}\left(\mathcal{M}/ \mathcal{M}^{G'}\right) \subset\; \overline{V}_{k}\quad for \;0\leq k \leq n-2\text{.}
\end{equation}
In particular, the quotient module
$\mathcal{M}/ \mathcal{M}^{G'}$ 
is supported by 
$V_{0}$ 
(the
Dirac module with support at the origin), then 
$\mathcal{M} = \mathcal{M}^{G'}$%
.

\section{Equivalence of categories\label{4}}

In this section, we establish the main result of this paper: Theorem \ref{MR}.
\newline
Recall that
$\overline{\mathcal{A}} =  \mathbb{C}\left\langle{f, \Delta, \theta, \Omega_{1}, \cdots ,\Omega_{r-1} }\right\rangle $
is the algebra of
$G'$%
-invariant differential operators. Since the Euler vector field
$\theta$
belongs to
$\overline{\mathcal{A}}$%
, we can decompose the algebra
$\overline{\mathcal{A}}$
under the adjoint action of
$\theta$:
\begin{equation}
\overline{\mathcal{A}} = \bigoplus\limits_{k \in \mathbb{N}}\overline{\mathcal{A}}\left[ k\right],\quad \overline{\mathcal{A}}\left[ k\right]=\{P\in \overline{\mathcal{A}}:\;\; \left[\theta , P\right] = kP\}
\end{equation}
and we can check that
\begin{equation}
\forall\; k,l \in \mathbb{N},\quad \overline{\mathcal{A}}\left[ k\right]\cdot \overline{\mathcal{A}}\left[ l\right]\subset \overline{\mathcal{A}}\left[ k+l\right].
\end{equation}
so
$\overline{\mathcal{A}}$
is a graded algebra.\newline
Recall also that
$\mathcal{J}\subset \overline{\mathcal{A}}$
is the annihilator of
$\mathbb{C}[f]$%
.
We have denoted
$\overline{\mathcal{J}}$
the preimage in
$\overline{\mathcal{A}}$
of the ideal in
$\overline{\mathcal{A}}/\mathcal{J}$
defined by the relations
(\ref{a0}), (\ref{b0}), (\ref{c0}), (\ref{d0})
of Proposition \ref{pu}:
\begin{eqnarray*}
\left[ \theta, f \right] & = & nf \text{,} \\
\left[\theta, \Delta \right] & = & -n \text{,}\Delta\\
f\Delta & = &  c\dfrac{\theta}{n}(\dfrac{\theta}{n}+\lambda_{1})\cdots(\dfrac{\theta}{n}+\lambda_{n-1}) \text{,} \\
\Delta{f} & = & c(\dfrac{\theta}{n}+1)(\dfrac{\theta}{n}+\lambda_{1}+1)\cdots(\dfrac{\theta}{n}+\lambda_{n-1} +1)\text{.} 
\end{eqnarray*}
We put
$\mathcal{A}$
the quotient of
$\overline{\mathcal{A}}$ by $\overline{\mathcal{J}}$:
$\mathcal{A} := \overline{\mathcal{A}}/\overline{\mathcal{J}}$
(see Corollary \ref{cu}).

Now, since
$\overline{\mathcal{J}}$
is an ideal of
$\overline{\mathcal{A}}$
it decomposes also under the adjoint action of
$\theta$:
\begin{equation}
\overline{\mathcal{J}} = \bigoplus\limits_{k \in \mathbb{N}}\overline{\mathcal{J}}\left[ k\right], \quad \overline{\mathcal{J}}\left[ k\right]=\overline{\mathcal{J}}\cap\overline{\mathcal{A}}\left[ k\right].
\end{equation}
Note that
$\overline{\mathcal{J}}$
is an homogeneous ideal of the graded algebra
$\overline{\mathcal{A}}$%
, thus the quotient algebra
$\mathcal{A} = \overline{\mathcal{A}}/\overline{\mathcal{J}}$
is naturally graded by
\begin{equation}
\mathcal{A}\left[k\right] :=\left(\overline{\mathcal{A}}/\overline{\mathcal{J}}\right)\left[k\right] =
\overline{\mathcal{A}}\left[k\right]/\overline{\mathcal{J}}\left[k\right]\text{.}
\end{equation}

\smallskip \smallskip

As in the introduction, we denote by
$\mathrm{Mod}^{\mathrm{gr}}(\mathcal{A})\;$
the category whose objects are finitely generated left
$\mathcal{A}$%
-modules 
$T$
such that for each 
$s \in T$%
, the 
$\mathbb{C}$%
-vector space spanned by the set
$\{\theta^{n}s\; /\; n\geq 1\}$
is finite dimensional. Equivalently
the category consisting of graded
$\mathcal{A}$%
-modules
$T$
of finite type such that
$\dim _{\mathbb{C}}\mathbb{C}\left[ \theta \right] u<\infty\; $
for any
$u$
in
$T$%
. In other words,
$T$
is a direct sum of finite dimensional
$\mathbb{C}$%
-vector
spaces:
\begin{equation}
T\;=\bigoplus\limits_{\alpha
\in \mathbb{C}}T_{\alpha },  \quad T_{\alpha }:=\bigcup\limits_{p\in \mathbb{N}}\ker \left( \theta -\alpha
\right) ^{p}\text{ (with }\dim _{\mathbb{C}}T_{\alpha }<\infty \text{)}
\end{equation}
equipped with the endomorphisms $f$%
, $\theta $%
, $\Delta $ of degree $%
n, $ $0,$ $-n$%
, respectively and satisfying the relations
(\ref{a0}), (\ref{b0}), (\ref{c0}), (\ref{d0})
of Proposition \ref{pu}
with $\left( \theta -\alpha
\right) $ being a nilpotent operator on each $T_{\alpha }$%
.\smallskip
\smallskip

Recall that
$\mathrm{Mod}_{\Lambda }^{\mathrm{rh}}(\mathcal{D}_{V})$
stands for the category consisting of holomorphic regular holonomic
$\mathcal{D}_{V}$%
-modules
whose characteristic variety \noindent is contained in
$\Lambda $
the union of conormal bundles to the orbits for the action of
$G$
on the complex vector space
$V$%
.\bigskip\

Let
$\mathcal{M}$
be an object in the category
$\mathrm{Mod}_{\Lambda }^{\mathrm{rh}}(\mathcal{D}_{V})$%
, denote by
$\Psi \left( \mathcal{M}\right) $
the submodule of
$\Gamma \left( V,\mathcal{M}\right) $
consisting of $G'$%
-invariant homogeneous global sections
$\;u$
in
$\mathcal{M}$
such that
$\dim_{\mathbb{C}}\mathbb{C}\left[ \theta \right] u<\infty $:
\begin{equation}
\Psi \left( \mathcal{M}\right) :=\left\{ u\in \Gamma \left( V,\mathcal{M}%
\right) ^{G'},\;\dim _{\mathbb{C}}\mathbb{C}\left[ \theta \right]
u<\infty \right\} \text{.}
\end{equation}
We are going to show that
$\Psi \left( \mathcal{M}\right)$
is an object in
$\mathrm{Mod}^{\mathrm{gr}}(\mathcal{A})$%
.\newline
Let
$\left( \sigma _{1,}\cdots ,\sigma _{p}\right) \in \Gamma \left( V,\mathcal{M}\right) ^{G'}$
be a finite family of homogeneous invariant global sections generating the
$\mathcal{D}_{V}$%
-module
$\Psi\left({\mathcal{M}}\right)$
(see Theorem \ref{td}):
\begin{equation}
\Psi \left( \mathcal{M}\right) :=\mathcal{D}_{V}\left\langle \sigma _{1,}\cdots ,\sigma _{p}\right\rangle\text{.}
\end{equation}
We are going to see that the family
$\left( \sigma _{1,}\cdots ,\sigma _{p}\right)$
generates also
$\Psi\left( \mathcal{M}\right)$
as an
$\mathcal{A}$%
-module:
indeed, an invariant section
$\sigma \in \Psi\left(\mathcal{M}\right)$
can be written as
\begin{equation}
 \sigma =\sum\limits_{j=1}^{p}q_{j}\left( X,D\right)\sigma _{j}\quad \text{where}\quad q_{j}\in \mathcal{D}_{V}\text{.}
\end{equation}
Let
$G_{c}$
be the compact maximal subgroup of
$G'$
and denote by
$\widetilde{q_{j}}:=\int_{G_{c}}g\cdot q_{j}dg$
the average of
$q_{j}$
over
$G_{c}$%
.
Then, the average
$\widetilde{q_{j}}$
belongs to the algebra
$\overline{\mathcal{A}}$
(i.e.,
 $\widetilde{q_{j}}\in \overline{\mathcal{A}}$%
).
Now, denote by
$f_{j}$
the class of
$\widetilde{q_{j}}$
modulo
$\overline{\mathcal{J}}$:
\begin{equation}
f_{j}\;:=\; \widetilde{q_{j}}\mod \overline{\mathcal{J}} \quad
\text{that is} \quad
\mathcal{\ }f_{j}\in \mathcal{A}.
\end{equation}
Therefore, we also have
\begin{equation}
\sigma =\sum\limits_{j=1}^{p}\widetilde{q_{j}}\sigma
_{j}=\sum\limits_{j=1}^{p}f_{j}\sigma _{j}\quad \text{with} \mathcal{\ }f_{j}\in \mathcal{A}\text{.}
\end{equation}
This last means that
\begin{equation}
\Psi \left( \mathcal{M}\right) :=\mathcal{A}\left\langle \sigma _{1,}\cdots ,\sigma _{p}\right\rangle,\label{graded1}
\end{equation}
and
$\Psi \left( \mathcal{M}\right) $
is an
$\mathcal{A}$%
-module.
Moreover, according to Theorem \ref{tu} ii), we have
 \begin{equation}
\Psi \left( \mathcal{M}\right) =\bigoplus\limits_{\alpha \in \mathbb{C}%
}\Psi \left( \mathcal{M}\right) _{\alpha }\label{graded2}
\end{equation}
where
\begin{equation}
\Psi \left( \mathcal{M}\right) _{\alpha }:=\left[ \Psi \left( \mathcal{M}%
\right) \right] \bigcap \left[ \bigcup\limits_{p\in \mathbb{N}}\ker
(\theta -\alpha )^{p}\right] \; \text{ (with }\dim _{\mathbb{C}}\Psi \left( \mathcal{M}\right) _{\alpha } <\infty \text{)}\label{graded3}
\end{equation}
is the finite dimensional
$\mathbb{C}$%
-vector
space of homogeneous global sections of degree
$\alpha\in \mathbb{C}$
in
$\Psi \left( \mathcal{M}\right) $%
.
Finally, we can check that
\begin{equation}
\mathcal{A}\left[{k}\right]\Psi \left( \mathcal{M}\right) _{\alpha } \subset
\Psi \left( \mathcal{M}\right) _{\alpha + k} \quad \;\text{for all} \; k \in \mathbb{N},\; \alpha \in \mathbb{C}\text{.}\label{graded4}
\end{equation}
So,
$\Psi \left( \mathcal{M}\right) $
is a graded
$\mathcal{A}$%
-module
of finite type for the Euler vector field
$\theta$
thanks to (\ref{graded1})-(\ref{graded4}).
This means that
$\Psi\left(\mathcal{M}\right)$
is an object in
$\mathrm{Mod}^{\mathrm{gr}}(\mathcal{A})$%
.\bigskip

Conversely, let $T$ be an object in the category $\mathrm{Mod}^{\mathrm{gr}}(%
\mathcal{A})$%
, one associates to it the $\mathcal{D}_{V}$%
-module
\begin{equation}
\Phi \left( T\right) :=\mathcal{M}_{0}\bigotimes_{\mathcal{A}}T
\end{equation}
where $\mathcal{M}_{0}:=\mathcal{D}_{V}/\overline{\mathcal{J}}$%
. Then $\Phi
\left( T\right) $ is an object in the category $\mathrm{Mod}_{\Sigma }^{%
\mathrm{rh}}(\mathcal{D}_{V})$%
.\smallskip\

Thus, we have defined two functors
\begin{equation}
\Psi :\mathrm{Mod}_{\Lambda }^{\mathrm{rh}}(\mathcal{D}_{V})\longrightarrow
\mathrm{Mod}^{\mathrm{gr}}(\mathcal{A})\text{, }\Phi :\mathrm{Mod}^{\mathrm{%
gr}}(\mathcal{A})\longrightarrow \mathrm{Mod}_{\Lambda }^{\mathrm{rh}}(%
\mathcal{D}_{V})\text{.}
\end{equation}

We need the two following lemmas:\smallskip \smallskip

\begin{lemma}
\label{ls}The canonical morphism
\begin{equation}
T\longrightarrow \Psi (\Phi \left( T\right) )\text{, }t\longmapsto 1\otimes t
\end{equation}
is an isomorphism, and defines an isomorphism of functors $\mathrm{Id}_{%
\mathrm{Mod}^{\mathrm{gr}}(\mathcal{A})}\longrightarrow \Psi \circ \Phi $%
.
\end{lemma}

\begin{proof}
We have set 
$\mathcal{M}_{0}:=\mathcal{D}_{V}/\overline{\mathcal{J}}$%
.
Denote by
$\varepsilon $
(the class of
$1_{\mathcal{D}}$
modulo
$\overline{\mathcal{J}}$%
) the canonical generator of
$\mathcal{M}_{0}$
. Recall that
$G_{c}$
is the compact maximal subgroup of
$G'$%
. Let
$h\in \mathcal{D}_{V} $
, denote by
$\widetilde{h}$
$\in \overline{\mathcal{A}}$
its average on
$G_{c}$
and by
$\varphi $
the class of
$\widetilde{h}$
modulo
$\overline{\mathcal{J}}$
, that is,
$\varphi \in $
$\mathcal{A}$%
. \newline
Since
$\varepsilon $
is
$G'$%
-invariant, we get
$\widetilde{h\varepsilon }$
$=\widetilde{h}\varepsilon =\varepsilon \varphi $
. Moreover, we have
$\widetilde{h}\varphi =0$
if and only if
$\widetilde{h}\in \overline{\mathcal{J}}$
, in other words
$\varphi =0$
. Therefore, the average operator (over
$G_{c}$%
)
$\mathcal{D}_{V}$
$\longrightarrow \overline{\mathcal{A}},$
$h\longmapsto \widetilde{h}$
induces a surjective
morphism of
$\mathcal{A}$%
-modules
$v:\mathcal{M}_{0}\longrightarrow \mathcal{A}$
. More generally, for any
$\mathcal{A}$%
-module
$T$
in the category
$\mathrm{Mod}^{\mathrm{gr}}(\mathcal{A})$
the morphism
$v\otimes 1_{T}$ is
surjective
\begin{equation}
v_{T}:\mathcal{M}_{0}\bigotimes_{\mathcal{A}}T\longrightarrow \mathcal{A}%
\bigotimes_{\mathcal{A}}T=T
\end{equation}
which is the left inverse of the morphism
\begin{equation}
u_{T}:T\longrightarrow \mathcal{M}_{0}\bigotimes_{\mathcal{A}}T\text{, }%
t\longmapsto \varepsilon \otimes t\text{,}
\end{equation}
that is,
$\left( v\otimes 1_{T}\right) \circ \left( \varepsilon \otimes
1_{T}\right) =v\left( \varepsilon \right) =1_{T}$
. This means that the
morphism
$u_{T}$
is injective. Next, the image of
$u_{T}$
is exactly the set of invariant sections of
$\mathcal{M}_{0}\bigotimes\limits_{\mathcal{A}}T=\Phi \left( T\right) $%
, that is,
$\Psi (\Phi (T))$
: indeed if
$\sigma=\sum\limits_{i=1}^{p}h_{i}\otimes t_{i}$
is an invariant section in
$\mathcal{M}_{0}\bigotimes\limits_{\mathcal{A}}T$
, we may replace each
$h_{i} $
by its average
$\widetilde{h_{i}}\in \mathcal{A}$
, then we get
\begin{equation}
\sigma =\sum\limits_{i=1}^{p}\widetilde{h_{i}}\otimes t_{i}=\varepsilon
\otimes \sum\limits_{i=1}^{p}\widetilde{h_{i}}t_{i}\;\in \;\varepsilon
\otimes T\text{,}
\end{equation}
that is,
$\sum\limits_{i=1}^{p}\widetilde{h_{i}}t_{i}\in T$
. Therefore, the
morphism
$\ u_{T}$
is an isomorphism from
$T$
to
$\Psi \left( \Phi \left(T\right) \right) $
and defines an isomorphism of functors.
\end{proof}

\medskip \medskip Next, we note the following:

\begin{lemma}
\label{lh}The canonical morphism
\begin{equation}
w:\Phi \left( \Psi \left( \mathcal{M}\right) \right) \longrightarrow
\mathcal{M}
\end{equation}
is an isomorphism and defines an isomorphism of functors
$\Phi \circ \Psi\longrightarrow \mathrm{Id}_{\mathrm{Mod}_{\Sigma }^{\mathrm{rh}}(\mathcal{D}_{V})}$%
.
\end{lemma}

\begin{proof}
As in the theorem \ref{td}, the
$\mathcal{D}_{V}$%
-module
$\mathcal{M}$
is generated by a finite family of invariant sections
$\left( \sigma_{i}\right) _{i=1,\cdots ,p}\in \Psi \left( \mathcal{M}\right) $
so that the morphism
$w$
is surjective. Now, consider
$\mathcal{Q}$
the kernel of the morphism
$w:\Phi \left( \Psi \left( \mathcal{M}\right) \right)\longrightarrow \mathcal{M}$
. It is also generated over
$\mathcal{D}_{V}$
by its invariant sections , that is, by
$\Psi \left( \mathcal{Q}\right) $%
. Then
we get
\begin{equation}
\Psi \left( \mathcal{Q}\right) \subset \Psi \left[ \Phi \left( \Psi \left(
\mathcal{M}\right) \right) \right] =\Psi \left( \mathcal{M}\right)
\end{equation}
where we used $\Psi \circ \Phi =Id_{\mathrm{Mod}^{\mathrm{gr}}(\mathcal{A})}$
(see the preceding Lemma \ref{ls}). Since the morphism $\Psi \left( \mathcal{%
M}\right) \longrightarrow \mathcal{M}$ is injective (%
$\Psi \left( \mathcal{M}%
\right) \subset \Gamma \left( V,\;\mathcal{M}\right) $%
), we obtain
$\Psi
\left( \mathcal{Q}\right) =0$%
. Therefore
 $\mathcal{Q}=0$ 
(because $\Psi
\left( \mathcal{Q}\right) $ generates $\mathcal{Q}$%
).
\end{proof}

\medskip\

\noindent This section ends by Theorem \ref{MR} established by means of the
preceding lemmas.\smallskip \smallskip

\begin{theorem}\label{MR}
Let 
$(G,V)$ 
be a representation of Capelli type with a one-dimensional quotient. Then
the functors $\Phi $ and $\Psi $ induce equivalence of categories 
\begin{equation}
\mathrm{Mod}_{\Lambda}^{\mathrm{rh}}(\mathcal{D}_{V})\;\overset{\sim }{%
\longrightarrow }\mathrm{Mod}^{\mathrm{gr}}(\mathcal{A}).
\end{equation}
\end{theorem}
\bigskip

\begin{center}
{\textbf{\underline{Appendix}}}
\end{center}

\begin{center}
\underline{\textbf{Appendix A: Representations of Capelli type with one-dimensional quotient}}

\begin{equation*}
\begin{array}{cccc}

         & \underline{(G,\; V)} \qquad & \underline{\deg{f}}\qquad & \qquad \underline{ b(s)}\\
         \\
(1)\quad &   (SO(n)\times\mathbb{C}^{*},\; \mathbb{C}^{n}) \qquad &  2 \qquad & (s+1)(s+\frac{n}{2}) \\
\\
(2)\quad &  (GL(n),\; S^{2}\mathbb{C}^{n} ) \qquad & n \qquad & \prod_{i=1}^{n}(s +\frac{i+1}{2})\\
\\
(3)\quad & (GL(n),\; \Lambda^{2} \mathbb{C}^{n}),\; n\; \text{even} \qquad  & \frac{n}{2} \qquad & \prod_{i=1}^{n}(s +2i-1)\\
\\
(4)\quad & (GL(n)\times{SL(n)},\; M_{n}(\mathbb{C})) \qquad & n \qquad & \prod_{i=1}^{n}(s +i)\\
\\
(5)\quad & (Sp(n)\times GL(2), \; \left(\mathbb{C}^{2n}\right)^{2}) \qquad & 2 \qquad & (s+1)(s+2n)\\
\\
(6)\quad & (SO(7)\times \mathbb{C}^{*}, \; spin = \mathbb{C}^{8}) \qquad & 2 \qquad & (s+2)(s+4)\\
\\
(7)\quad & (G_{2}\times \mathbb{C}^{*}, \;\mathbb{C}^{7}) \qquad & 2 \qquad & (s + 1)(s + \frac{7}{2})\\
\\
(8)\quad & (GL(4)\times Sp(2), M_{4}(\mathbb{C})) \qquad &  4 \qquad & (s+1)(s+2)(s+3)(s+4)
\end{array}
\end{equation*}
\newline

\underline{\textbf{Appendix B: Generic isotropy subgroups 
$G_{X_{0}}$ 
for representations of Capelli type}}

\begin{equation*}
\begin{array}{cccc}

         & \underline{(G,\; V)}  & \qquad \underline{G_{X_{0}}:= \text{isotropy subgroup at generic point} \;X_{0}
\in V\backslash {f^{-1}(0)}}\\
         \\
(1)\quad &   (SO(n)\times\mathbb{C}^{*},\; \mathbb{C}^{n})  \quad & SO(1)\times SO(n-1) \\
\\
(2)\quad &  (GL(n),\; S^{2}\mathbb{C}^{n} )  \quad & O(n)\\
\\
(3)\quad & (GL(n),\; \Lambda^{2} \mathbb{C}^{n}),\; n\; \text{even}  \quad & Sp(\frac{n}{2})\\
\\
(4)\quad & (GL(n)\times{SL(n)},\; M_{n}(\mathbb{C})) \quad & Sp(1)\times Sp(n-1)\\
\\
(5)\quad & (Sp(n)\times GL(2), \; \left(\mathbb{C}^{2n}\right)^{2}) \quad & SL(n)\\
\\
(6)\quad & (SO(7)\times \mathbb{C}^{*}, \; spin = \mathbb{C}^{8}) \qquad & SO(1)\times SO(6)\\
\\
(7)\quad & (G_{2}\times \mathbb{C}^{*},\; \mathbb{C}^{7})  \quad & \\
\\
(8)\quad & (GL(4)\times Sp(2), M_{4}(\mathbb{C}))   \quad &
\end{array}
\end{equation*}

(see A. Sasada \cite[(1), (2), (3), (13), (15)  p. 79-83]{Sa}
or Sato-Kimura
\cite[(1), (2), (3), (13), (15), p. 144-145]{S-K})
\newline

\underline{\textbf{Appendix C: Generic isotropy subgroups 
$H$ 
for derived subgroups 
$G'$ 
of the group 
$G$}}

\begin{equation*}
\begin{array}{ccc}

         & \underline{(G',\; V)}   & \underline{H = \text{isotropy subgroup at a generic point} 
 \;X_{0} \in V\backslash {f^{-1}(0)} }\\
         \\
(1)\quad &   (SO(n),\; \mathbb{C}^{n})   \quad & SO(1)\times SO(n-1) \\
\\
(2)\quad &  (SL(n),\; S^{2}\mathbb{C}^{n} ) \quad & SO(n)\\
\\
(3)\quad & (SL(n),\; \Lambda^{2} \mathbb{C}^{n}),\; n\; \text{even}   \qquad & Sp(\frac{n}{2})\\
\\
(4)\quad & (SL(n)\times{SL(n)},\; M_{n}(\mathbb{C}))  \quad & Sp(1)\times Sp(n-1)\\
\\
(5)\quad & (Sp(n)\times SL(2), \; \left(\mathbb{C}^{2n}\right)^{2})  \quad & SL(n)\\
\\
(6)\quad & (SO(7), \; spin = \mathbb{C}^{8})  \quad & SO(1)\times SO(6)\\
\\
(7)\quad & (G_{2}, \;\mathbb{C}^{7})  \qquad & \\
\\
(8)\quad & (SL(4)\times Sp(2), M_{4}(\mathbb{C}))   \quad &
\end{array}
\end{equation*}

\end{center}
\medskip

\noindent\textbf{Remark.}
The generic isotropy\footnote{a generic isotropic subgroup of
$G'$
is a stationary subgroup
at a generic point 
$X_{0}\in V \backslash S$
with
$S: f= 0$%
.} subgroups 
$H$
of
$(G', V)$
are connected.
\newline
\bigskip

\noindent \textbf{Aknowledgements: }  
The essential part of this paper was completed while
the author was a visiting member first at Max-Planck Institute for Mathematics (MPI), and then
at TATA Institute of Fundamental Research (TIFR). The author would like to express his
heartiest thanks to all these institutions and their members for their hospitalities. Specially,
we thank N. Nitsure for helpful discussions.

\noindent Philibert Nang\\

\noindent\'Ecole Normale Sup\'erieure,\newline
Laboratoire de Recherche en Math\'{e}matiques\newline
BP 8637\newline
Libreville, Gabon\\

\noindent E-mail: nangphilibert@yahoo.fr, pnang@ictp.it

\end{document}